\documentclass[12pt]{article}
\usepackage[T1]{fontenc}
\usepackage{amssymb,amsmath,amsthm}
\usepackage{enumerate,xspace}

\theoremstyle{definition}
\newtheorem*{definition}{Definition}

\theoremstyle{plain}
\newtheorem{remark}{Remark}[section]
\newtheorem{corollary}[remark]{Corollary}
\newtheorem{lemma}[remark]{Lemma}
\newtheorem{theorem}[remark]{Theorem}

\DeclareMathOperator{\cf}{cf}
\DeclareMathOperator{\Th}{Th}
\DeclareMathOperator{\sa}{SA}
\DeclareMathOperator{\sk}{SK}

\newcommand{\qcf}{Q^{\cf}}
\newcommand{\ebg}{\exists^{\,\cf}} 
\newcommand{\fgg}{\forall^{\,\cf}} 
\newcommand{\anti}{^\mathrm{anti}}

\renewcommand{\phi}{\varphi}

\begin{document}

\title{An exposition of the compactness of $L(\qcf)$ }

\date{September 3, 2020}

\author{Enrique Casanovas and Martin Ziegler\thanks{ Both authors were
    partially funded by a Spanish government grant MTM2017-86777-P.
    The first author also by a Catalan DURSI grant 2017SGR-270.} }

\maketitle

\begin{abstract}
  We give an exposition of the compactness of $L(\qcf_C)$, for any set $C$
  of regular cardinals.
\end{abstract}

\section{Introduction}

We present here a new and short exposition of the proof of the
compactness of the logic $L(\qcf_C)$, first-order logic extended by
the cofinality quantifier $\qcf_C$, where $C$ is a class of regular
cardinals. The logic and the proof of compactness are due to
S.~Shelah. The Compactness Theorem was stated and proved
in~\cite{Sh43}, but this article is not self-contained and some
fundamental steps of the proof must be found in the earlier
article~\cite{Sh:18}. The interested reader consulting these two
articles will soon realise that the structure of the proof is not
completely transparent and that to fully understand the details
requires a lot of work.

The most popular case of the cofinality quantifier is the logic
$L(\qcf_\omega)$ of the quantifier of cofinality $\omega$, that is,
$C=\{\omega\}$. Our motivation comes from the application of
$L(\qcf_\omega)$ in~\cite{CaSh:1116} to an old problem on
expandability of models. An anonymous referee of a preliminary version
of~\cite{CaSh:1116} did not accept the validity (in ZFC) of the
compactness proof presented in~\cite{Sh43}, apparently confused by the
assumption of the existence of a weakly compact cardinal made at the
beginning of the article. The assumption only applies to a previous
result on a logic stronger than first-order logic even for countable
models.

Our proof of compactness of $L(\qcf_C)$ uses some ideas
of~\cite{Sh43}, but it is more in the spirit of Keisler's proof
in~\cite{Kei70} of countable compactness of the logic $L(Q_1)$ with
the quantifier of uncountable cardinality. However we use a simpler
notion of weak model. J.~V\"a\"an\"anen in the last chapter
of~\cite{Van11} offers also a proof of compactness of $L(\qcf_\omega)$
in Keisler's style, but it is incomplete and only gives countable
compactness (see I.~Hodkinson's review in~\cite{Hodkinson12}).

There are some other proofs in the literature, but also
unsatisfactory. The proof by H-D.~Ebbinghaus in~\cite{Ebb85}, based on
a set-theoretical translation, is just an sketch and the proof of
J.A.~Makowsky and S.~Shelah in \cite{MakowskiShelah81} only replaces
part of Shelah's argument in~\cite{Sh43} by a different reasoning and
does not include all details.

\section{Connections}
For a linear ordering $(X,<)$ we use the expressions
\[\ebg x\; A(x),\;\;\text{and}\;\;\fgg x\;A(x)\]
for $\forall x'\,\exists x\; ( x' \leq x\land A(x))$, and $\exists
x'\,\forall x\; ( x' \leq x\to A(x))$, respectively.  The variables $x,x^\prime$ will range over the set $X$, $y,y^\prime$ over $Y$ and $z,z^\prime$ over $Z$.

\begin{definition}
  Let $X$ and $Y$ be two linear orderings. A connection between $X$
  and $Y$ is a relation $G\subset X\times Y$ with satisfies
\begin{align}
  \ebg x\, \fgg y\; G(x,y)&\;\;\text{and}\\
  \ebg y\, \fgg x\; \neg G(x,y).
\end{align}
\end{definition}
Note that $X$ and $Y$ cannot be connected if $X$ or $Y$ has a last
element.
\begin{remark}
  \begin{enumerate}
  \item If $X$ has no last element, the relation $x\leq y$ connects
    $X$ with itself.
  \item If $G$ connects $X$ and $Y$, then $\neg G^{-1}=\{(y,x)\mid
    \neg G(x,y)\}$ connects $Y$ and $X$.
  \item If $G$ connects $X$ and $Y$, and $H$ connects $Y$ and $Z$,
    then
    \[K=\Bigl\{(x,z)\,\Bigm|\,\exists y'\;\bigl(
    \forall y\; ( y'\leq y\to G(x,y))\;\land\;H(y',z)\bigr)\Bigr\}\]
    connects $X$ and $Z$.
  \end{enumerate}
\end{remark}
\begin{proof}
  1. and 2. are easy to see. We will not use 3. and leave the proof to
  the reader.
\end{proof}
\begin{remark}\label{R:anti}
  If $X$ and $Y$ are connected by $G$,
  then they are  also connected by
  \[G'=\Bigl\{(x,y)\Bigm| \exists x'\;\bigl(x\leq x'\land
  \forall y'\; (y\leq y'\to G(x',y'))\bigr)\Bigr\}.\]
  $G'$ is antitone in $x$ and monotone in $y$.
\end{remark}
\begin{proof}
  It is easy to see that
  $G\anti=\bigl\{(x,y)\bigm| \exists x'\;(x\leq x'\land
  G(x',y))\bigr\}$ connects $X$ and $Y$ and is antitone in $x$. Now
  it can be seen that
  \[G'=(\neg((\neg G^{-1})\anti)^{-1})\anti.\]
\end{proof}
\begin{lemma}\label{L:gleiche_cf}
  Two linear orders without last element are connected if and only if
  they have the same cofinality.
\end{lemma}
\begin{proof}
  If $\cf(X)=\cf(Y)=\kappa$, choose two increasing cofinal sequences
  $(x_\alpha\mid\alpha<\kappa)$ and $(y_\alpha\mid\alpha<\kappa)$ in
  $X$ and $Y$. Then \[G=\{(x,y)\mid\exists\alpha\;(x\leq x_\alpha\land
  y_\alpha\leq y)\}\] connects $X$ and $Y$.\footnote{It suffices to
    assume that the $y_\alpha$ are increasing. Also one can use
    $G=\{(x_\alpha,y)\mid y_\alpha\leq y)\}$.}

  For the converse assume that $\cf(X)=\kappa$, and that $G$ connects
  $X$ and $Y$. Choose a cofinal sequence $(x_\alpha\mid\alpha<\kappa)$
  in $X$ and elements $y_\alpha$ in $Y$ such that $y_\alpha\leq y\to
  G(x_\alpha,y)$ for all $y$. Then the $y_\alpha$ are cofinal in $Y$.
  To see this we use that there are cofinally many $y$ such that $\neg
  G(x,y)$ for sufficiently large $x$, which implies that $\neg
  G(x_\alpha,y)$ for some $\alpha$. This implies $y<y_\alpha$.
\end{proof}
\begin{lemma}\label{L:unschoen}
  Assume that $G\subset X\times Y$  satisfies
  \begin{align}
    &\ebg x\,\exists y\; G(x,y)\;\;\text{and}\label{L:unschoen:xy}\\
    &\forall y'\,\exists x'\,\forall xy\;\bigl((x'\leq x\land y\leq y')\to\neg
    G(x,y)\bigr).\label{L:unschoen:yx}
  \end{align}
  Then $G'=\{(x,y)\mid\exists y'\;(y'\leq y\land G(x,y'))\}$ connects
  $X$ and $Y$.
\end{lemma}
\noindent Note that a connecting $G$ which is monotone in $y$ satisfies
\eqref{L:unschoen:xy} and \eqref{L:unschoen:yx}.
\begin{proof}
  This is a straightforward verification.
\end{proof}

\section{The Main Lemma}

Consider a $L$-structure $M$ with two (parametrically) definable
linear orderings, $<_\phi$ and $<_\psi$ of its universe, both without
last element. We say that $\phi$ and $\psi$ are \emph{definably
  connected} if there is a definable connection between $(M,<_\phi)$
and $(M,<_\psi)$.

Recall that a formula $\varphi(x)$ \emph{isolates} a partial type
$\Sigma(x)$ in a theory $T$ if it is consistent with $T$ and implies
$\Sigma(x)$ in $T$ (see Definition 4.1.1 in~\cite{TentZiegler10} or
the definition of locally realizing a type in~\cite{Cha-Kei90}). $T$
isolates $\Sigma(x)$ if some formula $\varphi(x)$ does it in $T$.

\begin{lemma}\label{L:Main}
  If $\phi$ and $\psi$ are not definably connected, and $c$ is a new
  constant, the theory
  \[T'=\Th(M,m)_{m\in M}\cup\{m<_\phi c\mid m\in M\}\]
  does not isolate the partial type $\Sigma(y)= \{n<_\psi y\mid n\in
  M\}$.
\end{lemma}
\begin{proof}
  Assume that $\gamma(c,y)$, for some $L(M)$-formula $\gamma(x,y)$,
  isolates $\Sigma(y)$ in $T'$. This means that
  \begin{enumerate}
  \item $T'\cup\{\gamma(c,y)\}$ is consistent.
  \item $T'\vdash \gamma(c,y)\to n<_\psi y$ for all $n\in M$.
  \end{enumerate}
  We show that the relation $G$ defined by $\gamma(x,y)$ has
  properties \eqref{L:unschoen:xy} and \eqref{L:unschoen:yx} of Lemma
  \ref{L:unschoen}, where $X=(M,<_\phi)$ and $Y=(M,<_\psi)$. This will
  contradict the hypothesis of our Lemma.

  That $T'\cup\{\gamma(c,y)\}$ is consistent means that for all $m\in
  M$ the theory $\Th(M,m)_{m\in M}$ does not prove $m\leq_\phi
  c\to\neg\exists y\;\gamma(c,y)$, which means that $M\models\exists
  x(m\leq_\psi x\land\exists y\;\gamma(x,y))$. This is exactly
  condition \eqref{L:unschoen:xy} of \ref{L:unschoen}.

  That $T'\vdash \gamma(c,y)\to n<_\psi y$ means that there is an
  $m\in M$ such that $\Th(M,m)_{m\in M}$ proves $(m\leq_\phi c\land
  \gamma(c,y))\to n<_\psi y$, which means $M\models\forall xy\;
  \bigl(( m\leq_\phi x\land y\leq_\psi n)\to\neg\gamma(x,y)\bigr)$.
  The existence of such $m$ for all $n$ is exactly condition
  \eqref{L:unschoen:yx} of \ref{L:unschoen}.
\end{proof}

\begin{corollary}\label{C:Omitting2} Assume $\kappa$ is regular,
$|M|,|L|\leq \kappa$, and $<_\varphi$ is a definable linear ordering
  of $M$ without last element. Then there is an elementary extension
  $N$ of $M$ such that:
  \begin{enumerate}
  \item $M$ is not $<_\varphi$-cofinal in $N$.
  \item If $<_\psi$ is a definable linear ordering of $M$ of
    cofinality $\kappa$, and $\psi$ and $\varphi$ are not definably
    connected, then $M$ is $<_\psi$-cofinal in $N$.
  \end{enumerate}
\end{corollary}
\begin{proof}
  Let $c$ be a new constant and let $T^\prime=\Th(M,m)_{m\in M}\cup
  \{m<_\varphi c\mid m\in M\}$. By Lemma~\ref{L:Main}, $T^\prime$ does
  not isolate any of the types $\Sigma_\psi(y)= \{n<_\psi y \mid n\in
  M\}$. Each $\Sigma_\psi(y)$ consists of a $<_\psi$-ordered chain of
  formulas increasing in strength. So by regularity of $\kappa$, for
  any $<_\psi$ of cofinality $\kappa$ the type $\Sigma_\psi(y)$ cannot
  be isolated neither by means of a set of $<\kappa$ formulas. By the
  $\kappa$-Omitting Types Theorem (see Theorem 2.2.19
  in~\cite{Cha-Kei90}), there is a model of $T^\prime$ omitting all
  types $\Sigma_\psi(y)$ for any $<_\psi$ of cofinality $\kappa$. This
  gives the elementary extension $N$.
\end{proof}

This corollary applies in particular to the case $\kappa=\omega$. Here
the assumption on the cofinality of $<_\psi$ is not needed since it is
the only possible cofinality in a countable model, and the Omitting
Types Theorem used in the proof is the ordinary one for countable
languages and countably many non-isolated types.

\section{Completeness}

For a language $L$ let $L(\qcf)$ be the set of formulas which are
built like first-order formulas but using an additional two-place
quantifier $\qcf xy\;\phi$, for different variables $x$ and $y$. Let
$C$ be class a of regular cardinals and $M$ an $L$-structure. For a
binary relation $R$ on $M$, we write ``$\cf R \in C$ '' for ``$R$ is a
linear ordering of $M$, without last element and cofinality in $C$ ''.

The satisfaction relation $\models_C$ for $L$-structures $M$,
$L(\qcf)$-formulas $\psi(\bar z)$, and tuples $\bar c$ of elements of
$M$ is defined inductively, where the $\qcf$-step is
\[
  M\models_C\qcf xy\;\phi(x,y,\bar
  c)\;\Leftrightarrow\;\cf\,\{(a,b)\mid M\models_C\phi(a,b,\bar
  c)\}\in C.\]

\noindent We say that $M$ is a \emph{$C$-model of $T$}, a set of
$L(\qcf)$-sentences, if $M\models_C\psi$ for all $\psi\in T$.

A \emph{weak} structure $M^\ast=(M,\ldots)$ is an $L^\ast$-structure,
where $L^*$ is an extension of $L$ by an $n$-ary relation $R_\phi$ for
every $L(\qcf)$-formula $\phi(x,y,z_1,\dotsc,z_n)$. Satisfaction is
defined using the rule

\[M^\ast\models\qcf xy\;\phi(x,y,\bar c)\;\Leftrightarrow
\;M^\ast\models R_\phi(\bar c).\] In weak structures every
$L(\qcf)$-formula is equivalent to a first-order $L^\ast$-formula, and
conversely. So the $L(\qcf)$-model theory of weak structures is the
same as their first-order model theory.

Note that  the $C$-semantics of $M$ is given by the semantics of
the weak structure $M^\ast$\/ if one sets \[M^\ast\models R_\phi(\bar
c)\;\Leftrightarrow\;M\models_C\qcf xy\;\phi(x,y,\bar c).\]

The following lemma is clear:
\begin{lemma}\label{L:weak=C}
  The $C$-semantics of $M$ is given by the weak structure $M^\ast$ if and
  only if
  \[M^\ast\models\qcf xy\;\phi(x,y,\bar
  c)\;\Leftrightarrow\;\cf\,\{(a,b)\mid M^\ast\models\phi(a,b,\bar
  c)\}\in C\] for all $\phi$ and $\bar c$.
\end{lemma}

The following property of weak structures $M^\ast$ can be expressed by
a set $\sa(L)$ of $L(\qcf)$ sentences (the Shelah Axioms):\\

\parbox{30em}{ \emph{If the $L(\qcf)(M)$-formula $\phi(x,y)$ satisfies
    $M^\ast\models\qcf xy\;\phi(x,y)$ then $\phi$ defines a linear
    ordering $<_\phi$ without last element. Furthermore, if
    $\psi(x,y)$ defines a linear ordering $<_\psi$ and
    $M^\ast\models\neg\qcf xy\;\psi(x,y)$, there is no definable
    connection between $(M,<_\phi)$ and $(M,<_\psi)$.}}

\begin{lemma}\label{L:qc_sa}
  $L$-structures with the $C$-semantics are models of $\sa(L)$.
\end{lemma}
\begin{proof}
  This follows from Lemma \ref{L:gleiche_cf}.
\end{proof}

\begin{theorem}\label{T}
  Let $C$ be a non-empty class of regular cardinals, different from
  the class of all regular cardinals. An $L(\qcf)$-theory $T$ has a
  $C$-model if and only if $T\cup\sa(L)$ has a weak model.
\end{theorem}
\begin{proof}
  One direction follows from Lemma \ref{L:qc_sa}. For the other
  direction assume that $T\cup\sa(L)$ has a weak model.\\

  \noindent Claim 1: If $L$ is countable, $T$ has a an
  $\{\omega\}$-model of cardinality $\omega_1$.\\

  \noindent Proof. Let $M_0^\ast$ be a countable weak model of
  $T\cup\sa(L)$. Consider a linear ordering $<_\phi$ without last
  element and $M_0^\ast\models\neg\,\qcf xy\,\phi$. Then by Corollary
  \ref{C:Omitting2} for $\kappa=\omega$ and the axioms $\sa(L)$, there
  is an elementary extension $M_1^\ast$ such that $M_0$ is not
  $<_\phi$-cofinal in $M_1$, but $<_\psi$-cofinal is in $M_1$ for
  every $\psi$ with $M_0^\ast\models\qcf xy\,\psi$. We may assume that
  $M^\ast_1$ is countable. Continuing in this manner, taking unions at
  limit stages, one constructs an elementary chain of countable weak
  models $M^\ast_0\prec M^\ast_1\dotsb$ of length $\omega_1$ with
  union $M^\ast$, such that
  \begin{enumerate}
  \item If $<_\phi$ is a linear ordering of $M^\ast$ without last
    element and $M^\ast\models\neg\,\qcf xy\,\phi$, and if the
    parameters of $\phi$ are in $M_\alpha$, then for uncountably many
    $\beta\geq\alpha$, $M_\beta$ is not $<_\phi$-cofinal in
    $M_{\beta+1}$.
  \item If $M^\ast\models\qcf xy\,\psi$, and the parameters of $\psi$
    are in $M_\alpha$, then $M_\alpha$ is $<_\psi$-cofinal in $M$.
  \end{enumerate}
  It follows that, if $M^\ast\models\neg\,\qcf xy\,\phi$, then either
  $\phi$ does not define a linear ordering without last element, or
  $<_\phi$ has cofinality $\omega_1$. And, if $M^\ast\models\,\qcf
  xy\,\psi$, then $<_\psi$ has cofinality $\omega$. By Lemma
  \ref{L:weak=C} $M$ is an $\{\omega\}$-model of the $L(\qcf)$-theory
  of $M^\ast$, and whence an $\{\omega\}$-model of $T$. This proves
  Claim 1.\\

  Let $L'$ be the extension of $L$ which has for every
  $L(\qcf)$-formula $\phi(x,y,\bar z)$ a new relation symbol $V_\phi$
  of arity $2+2\cdot|\bar z|$. Let $\sk$ be the set of axioms which
  state that if $\phi(x,y,\bar c_1)$ and $\phi(x,y,\bar c_2)$ define
  linear orderings without last elements, and \[\qcf xy\,\phi(x,y,\bar
  c_1)\;\leftrightarrow\; \qcf xy\,\phi(x,y,\bar c_2),\] then
  $V_\phi(x,y,\bar c_1,\bar c_2)$ defines a connection between the two
  orderings.\\

  \noindent Claim 2: $T\cup\sa(L')\cup\sk$ has a weak model.\\

  \noindent Proof: By compactness we may assume that $L$ is countable.
  Then $T$ has an $\{\omega\}$-model $M$ of cardinality $\omega_1$, by
  Claim 1. If the $L(\qcf)$-formula $\phi(x,y,\bar c_1)$ and
  $\phi(x,y,\bar c_2)$ define linear orderings without last element,
  and $M\models_C\qcf xy\,\phi(x,y,\bar c_1)\,\leftrightarrow\, \qcf
  xy\,\phi(x,y,\bar c_2)$, then the two orderings have the same
  cofinality, namely $\omega$ or $\omega_1$, and there is a connection
  between them by Lemma \ref{L:gleiche_cf}. We interpret
  $V_\phi(x,y,\bar c_1,\bar c_2)$ by any such connection. This yields
  an expansion of $M$, which is an $\{\omega\}$-model of
  $T\cup\sa(L')\cup\sk$. This proves Claim 2.\\

  To prove the theorem, we choose two regular cardinals
  $\lambda<\kappa$ such that $|L|\leq\kappa$ and either
  $\lambda\not\in C$ and $\kappa\in C$ or conversely. Let $M_0^\ast$
  be a weak model of $T\cup\sa(L')\cup\sk$. It $M_0^\ast$ is finite,
  it is a $C$-model of $T$ for trivial reasons\footnote{$\sa(L)$ is
    used here.}. Otherwise we may assume that $M_0^\ast$ has
  cardinality $\kappa$ and all $L(\qcf)$-definable linear orderings
  without last element have cofinality $\kappa$. Let us first assume
  that $\lambda\not\in C$ and $\kappa\in C$.

  Consider an $L(\qcf)$-definable linear ordering $<_\phi$ without
  last element and $M_0^\ast\models\neg\,\qcf xy\,\phi$. Then by
  Corollary \ref{C:Omitting2} and the axioms $\sa(L')$, there is an
  elementary extension $M_1^\ast$ such that $M_0$ is not
  $<_\phi$-cofinal in $M_1$, but $<_\psi$-cofinal in $M_1$ for every
  $L(\qcf)$-definable ordering with $M_0^\ast\models\qcf xy\,\psi$. So
  $(M_1,<_\psi)$ has still cofinality $\kappa$. We may assume that
  $M^\ast_1$ has cardinality $\kappa$. Then by the axioms $\sk$, every
  $L(\qcf)$-definable ordering $<_\psi$ of $M_1$ with
  $M_1^\ast\models\qcf xy\,\psi$ is connected to an $M_0$-definable
  ordering $<_{\psi_0}$, and so has cofinality $\kappa$.

  Continuing in this manner, taking unions at limit stages, one
  constructs an elementary chain of weak models $M^\ast_0\prec
  M^\ast_1\dotsb$ of length $\lambda$ with union $M^\ast$, such that
  \begin{enumerate}
  \item If $<_\phi$ is an $L(\qcf)$-definable linear ordering $<_\phi$
    of $M^\ast$ without last element and $M^\ast\models\neg\,\qcf
    xy\,\phi$, and if the parameters of $\phi$ are in $M_\alpha$, then
    for $\lambda$-many $\beta\geq\alpha$, $M_\beta$ is not
    $<_\phi$-cofinal in $M_{\beta+1}$.
  \item If $M^\ast\models\qcf xy\,\psi$, and the parameters of $\psi$
    are in $M_\alpha$, then $M_\alpha$ is $<_\psi$-cofinal in $M$.
  \end{enumerate}
  It follows that, if $M^\ast\models\neg\,\qcf xy\,\phi$, then either
  $\phi$ does not define a linear ordering without last element, or
  $<_\phi$ has cofinality $\lambda$. And, if $M^\ast\models\,\qcf
  xy\,\psi$, then $<_\psi$ has cofinality $\kappa$. By Lemma
  \ref{L:weak=C} $M\restriction L$ is an $C$-model of the
  $L(\qcf)$-theory of $M^\ast$, and whence a $C$-model of $T$.

  The proof in the case $\lambda\in C$ and $\kappa\not\in C$ is,
  mutatis mutandis, the same.
\end{proof}

\begin{corollary}
  For every class $C$ of regular cardinals, the logic $L(\qcf_C)$ is
  compact.
\end{corollary}

We  have always assumed that whenever $\qcf xy\varphi(x,y,\bar c)$,
the definable ordering $<_\varphi$ linearly orders the universe.
This is not exactly the assumption of Shelah in~\cite{Sh43}:
with his definition   $<_\varphi$ linearly orders
$\{x\mid  \exists  y \,\varphi(x,y,\bar c)\}$, the domain of $\varphi$.
The results presented here, in particular completeness and
compactness, also apply to this modification of the semantics,
it suffices to add, for each such $\varphi$,  new relation symbols
$R_\varphi$ and   $H_\varphi$, and declare that for every $\bar c$,
$R_\varphi(x,y,\bar c)$ defines a linear ordering $<^\prime_\varphi$
on the universe and  $H_\varphi(x,y,\bar c)$  connects  $<_\varphi$
and  $<^\prime_\varphi$. This gives compactness. For the
formulation of completeness (Theorem~\ref{T}) one must
adapt the axioms $\mathrm{SA}$ to the new situation.

\nocite{Sh43}
\nocite{MakowskiShelah81}
\nocite{Kei70}
\nocite{Van11}
\nocite{Sh:18}
\nocite{CaSh:1116}
\nocite{Ebb85}
\nocite{TentZiegler10}
\nocite{Cha-Kei90}

\bibliographystyle{acm}


\begin{thebibliography}{1}

\bibitem{CaSh:1116}
{\sc Casanovas, E., and Shelah, S.}
\newblock {Universal theories and compactly expandable models}.
\newblock To appear in \emph{The Journal of Symbolic Logic}, 2019.

\bibitem{Cha-Kei90}
{\sc Chang, C.~C., and Keisler, H.~J.}
\newblock {\em Model Theory}.
\newblock North Holland P.C., Amsterdam, third edition, 1990.


\bibitem{Ebb85}
{\sc Ebbinghaus, H.-D.}
\newblock {E}xtended logics: the general framework.
\newblock In {\em Model--Theoretic Logics}, J.~Barwise and S.~Feferman, Eds.
  Springer Verlag, 1985, pp.~25--76.

\bibitem{Hodkinson12}
{\sc Hodkinson, I.}
\newblock Book review - {M}odels and {G}ames by {J}. {V}\"a\"an\"anen.
\newblock {\em The Bulletin of Symbolic Logic 18\/} (2012), 406--408.

\bibitem{Kei70}
{\sc Keisler, H.~J.}
\newblock {L}ogic with the quantifier ``there exist uncountably many''.
\newblock {\em Annals of Mathematical Logic 1\/} (1970), 1--93.

\bibitem{MakowskiShelah81}
{\sc Makowsky, J.~A., and Shelah, S.}
\newblock The theorems of {B}eth and {C}raig in abstract model theory. {II}.
\newblock {\em Archiv f{\"u}r mathematische Logik und Grundlagenforschung 21\/}
  (1981), 13--35.

\bibitem{Sh:18}
{\sc Shelah, S.}
\newblock On models with power like orderings.
\newblock {\em The Journal of Symbolic Logic 37\/} (1972), 247--267.

\bibitem{Sh43}
{\sc Shelah, S.}
\newblock Generalized quantifiers and compact logic.
\newblock {\em Transactions of the American Mathematical Society 204\/} (1975),
  342--364.

\bibitem{TentZiegler10}
{\sc Tent, K.,  and Ziegler, M.}
\newblock {\em A course in Model Theory}, volume~40 of {\em Lecture Notes in
  Logic}.
\newblock Cambridge University Press, Cambridge, 2012.

\bibitem{Van11}
{\sc V{\"a}{\"a}n{\"a}nen, J.}
\newblock {\em Models and Games}, vol.~132 of {\em Cambridge Studies in
  Advanced Mathematics}.




\end{thebibliography}




\noindent{\sc
Departament de Matem\`atiques i Inform\`atica\\
Universitat de Barcelona}\\
{\tt e.casanovas@ub.edu}\\

\noindent{\sc
Mathematisches Institut \\
Universit\"at Freiburg\\
{\tt ziegler@uni-freiburg.de}

\end{document}
